\documentclass{amsart}
\usepackage{color}
\numberwithin{equation}{section}
\usepackage[all]{xy}
\usepackage{amssymb}
\usepackage{comment}
\let\cal\mathcal
\def\Ascr{{\cal A}}
\def\Bscr{{\cal B}}

\def\Dscr{{\cal D}}
\def\Escr{{\cal E}}

\def\Mscr{{\cal M}}
\def\Nscr{{\cal N}}
\def\Oscr{{\cal O}}
\def\Pscr{{\cal P}}
\def\Qscr{{\cal Q}}
\def\Rscr{{\cal R}}
\def\Sscr{{\cal S}}

\def\Zscr{{\cal Z}}
\let\blb\mathbb

\def \PP{{\blb P}}

\def \ZZ{{\blb Z}}

\def \NN{{\blb N}}

\def\Id{\operatorname{id}}
\def\pr{\mathop{\text{pr}}\nolimits}

\def\Lotimes{\overset{L}{\otimes}}

\def\mod{\operatorname{mod}}
\def\Gr{\operatorname{Gr}}

\def\coh{\mathop{\text{\upshape{coh}}}}

\def\Spec{\operatorname {Spec}}

\def\GL{\operatorname {GL}}

\def\Ext{\operatorname {Ext}}
\def\Hom{\operatorname {Hom}}
\def\uHom{\operatorname {\mathcal{H}\mathit{om}}}

\def\End{\operatorname {End}}
\def\RHom{\operatorname {RHom}}

\def\im{\operatorname {im}}

\def\ker{\operatorname {ker}}

\def\End{\operatorname {End}}

\def\rk{\operatorname {rk}}

\def\gldim{\operatorname {gl\,dim}}

\def\r{\rightarrow}

\newtheorem{lemma}{Lemma}[section]
\newtheorem{proposition}[lemma]{Proposition}
\newtheorem{theorem}[lemma]{Theorem}
\newtheorem{corollary}[lemma]{Corollary}

\newtheorem{lemmas}{Lemma}[subsection]

\newtheorem{theorems}[lemmas]{Theorem}

\newtheorem*{sublemma}{Sublemma}

\theoremstyle{definition}

\newtheorem{example}[lemma]{Example}

{

}

\theoremstyle{remark}

\newtheorem{remark}[lemma]{Remark}

\newdimen\uboxsep \uboxsep=1ex
\def\uboxn#1{\vtop to 0pt{\hrule height 0pt depth 0pt\vskip\uboxsep
\hbox to 0pt{\hss #1\hss}\vss}}

\def\uboxs#1{\vbox to 0pt{\vss\hbox to 0pt{\hss #1\hss}
\vskip\uboxsep\hrule height 0pt depth 0pt}}

\newcommand{\rn}{\gamma_{r,n}}

\def\Sym{\operatorname{Sym}} \def\Sp{\operatorname{Sp}}
 \def\codim{\operatorname{codim}}
 \let\oldmarginpar\marginpar
\def\marginpar#1{\oldmarginpar{\color{blue}\tiny #1}}
 
 \author{\v{S}pela \v{S}penko}
\email[\v{S}pela \v{S}penko]{Spela.Spenko@ed.ac.uk}
\address{School of Mathematics\\
  The University of Edinburgh\\
  James Clerk Maxwell Building\\
  The King's Buildings\\
  Peter Guthrie Tait Road\\
  EDINBURGH\\
  EH9 3FD \\
  Scotland, UK} \author{Michel Van den Bergh} \email[Michel Van den
Bergh]{michel.vandenbergh@uhasselt.be}
\address{Department of mathematics\\Universiteit Hasselt\\
  Martelarenlaan 42\\
  3500 Hasselt\\
  Belgium} 
\thanks{The first author was supported by the
  L'Or\'eal-UNESCO scholarship ``For women in science''.}
\thanks{The second author is a senior researcher at the
  Research Foundation Flanders} 
 \title[Comparing resolutions]{Comparing
  the commutative and non-commutative resolutions for determinantal
  varieties of skew symmetric and symmetric matrices}
\begin{document}
\begin{abstract} Let $Y$ be the variety of (skew) symmetric $n\times n$-matrices of rank $\le r$.  In paper we construct a full faithful
embedding between the derived category of a non-commutative resolution of $Y$, constructed earlier by the authors, and the 
derived category of the classical Springer resolution of $Y$.
\end{abstract}
\maketitle
\section{Introduction}
\label{ref-1-0}
Throughout $k$ is an algebraically closed field of characteristic zero. If $\Lambda$ is a right noetherian
ring then we write $\Dscr(\Lambda)$ for $D^b_f(\Lambda)$, the bounded derived category
of right $\Lambda$-modules with finitely generated cohomology. Similarly for a noetherian scheme/stack $X$ we write
$\Dscr(X):=D^b_{\coh}(X)$. 

\medskip

If $Y$ is the determinantal variety of $n\times n$-matrices of rank
$\le r$ then in \cite{VdB100} (and independently in
\cite{SegalDonovan}) a ``non-commutative crepant resolution'' \cite{Leuschke,VdB32}
$\Lambda$ for $k[Y]$ was constructed.  Such an NCCR is a $k[Y]$-algebra which has in
particular the property that $\Dscr(\Lambda)$ is a ``strongly crepant
categorical resolution'' of~$\operatorname{Perf}(Y)$ (the derived
category of perfect complexes on $Y$) in the sense of~\cite[Def.\
3.5]{Kuznetsov}. This NCCR was constructed starting from a tilting bundle on the standard Springer type resolution 
of singularities $Z\r Y$ where $Z$ is a vector bundle over a Grassmannian. Indeed 
the main properties of $\Lambda$  were derived from the existence of a derived equivalence
between $\Dscr(\Lambda)$ and $\Dscr(Z)$. 

\medskip

In this paper we discuss suitably adapted versions of these results for determinantal varieties of symmetric matrices and
skew symmetric matrices. It turns out that both settings are very similar but notationally cumbersome to
treat together. So we present our main results and arguments in the skew symmetric case. The modifications needed for the symmetric
case will be discussed briefly in Section~\ref{symsec}.

\medskip

Let $n>r>0$ with $2|r$ and now let $Y$ be the variety of skew symmetric $n\times n$-matrices of rank $\le r$. 
If $n$ is odd then in \cite{SpenkoVdB}  we constructed an NCCR
 $\Lambda$ for $k[Y]$ (the existence of the resulting strongly crepant categorical resolution of $Y$ was conjectured in \cite[Conj.\ 4.9]{Kuznetsov4}).
  The construction of $\Lambda$ also works when  $n$ is even but then
 $\Lambda$ is not an NCCR, albeit very close to one. In particular one may show that $\Dscr(\Lambda)$ is a
 ``weakly crepant categorical resolution'' of $\operatorname{Perf}(Y)$, again in the sense of \cite{Kuznetsov} (see \cite{Abuaf} for an entirely different construction of such resolutions).

\medskip

In contrast to \cite{VdB100,SegalDonovan} the construction of the NCCR
$\Lambda$ is based on invariant theory and does not use geometry.
Nonetheless it is well known that also in this case $Y$ has a canonical
(commutative) Springer type resolution of singularities $Z\r Y$ and our
main concern below will be the relationship between the
resolutions $\Lambda$ and $Z$.  In particular we will construct a
$k[Y]$-linear embedding
\begin{equation}
\label{ref-1.1-1}
\Dscr(\Lambda)\hookrightarrow \Dscr(Z).
\end{equation}
For $n$ odd such an inclusion is expected  by the fact that NCCRs are conjectured to yield 
minimal categorical resolutions. Note that the embedding \eqref{ref-1.1-1} turns out to be somewhat non-trivial. The image of $\Lambda$ is a coherent sheaf of $\Oscr_Z$-modules, but it is not a vector bundle.

\medskip

As already mentioned, the construction of $\Lambda$ uses invariant theory.  We explain this next. Let $H$, $V$ be
vector spaces of dimension $n$, $r$ with $V$ being in addition equipped
with a symplectic bilinear form $\langle-,-\rangle$. The corresponding symplectic group is denoted
by $\Sp(V)$. 

If ${{\chi}}$
is a partition with $l({{\chi}})\le r/2$  then we let $S^{\langle{{\chi}}\rangle}V$ be the
irreducible representation
of $\Sp(V)$ with highest weight~${{\chi}}$. If ${{\chi}}=({{\chi}}_1,\ldots,{{\chi}}_{r})\in\ZZ^r$ is a dominant $\GL(V)$-weight then we let $S^{{\chi}} V$ be the 
irreducible $\GL(V)$-representation with highest weight ${{\chi}}$.

Put
$
X=\Hom(H,V)
$ 
and let $T$ be the coordinate ring of $X$:
\[
T=\Sym_k(H\otimes_k V^\vee).
\]
Put 
\begin{equation}
\label{ref-1.2-2}
M({{\chi}}):=
(S^{\langle{{\chi}}\rangle}V\otimes_k T)^{\Sp(V)}\,.
\end{equation}
Thus $M(\chi)$ is a ``module of covariants'' in the sense of \cite{VdB9}.
Let $B_{m,n}$ be the set of partitions contained in a box with $m$
rows and $n$ columns. Put 
\begin{equation}
\label{ref-1.3-3}
M=\bigoplus_{{{\chi}}\in B_{r/2,\lfloor n/2\rfloor-r/2}} M({{\chi}})
\end{equation}
and 
$
\Lambda=\End_{R}(M)
$.
In \cite{SpenkoVdB} the following result (which improves on \cite{WeymanZhao}) was proved:
\begin{theorem} 
\label{ref-1.1-4} One has $\gldim\Lambda<\infty$. Moreover if $n$ is odd then 
$\Lambda$ is a Cohen-Macaulay $R:=T^{\Sp(V)}$-module. In other words, in the terminology of
\cite{Leuschke,VdB32}, 
when $n$ is odd $\Lambda$ 
is a non-commutative crepant resolution (NCCR) of $R$.
\end{theorem}
By the first fundamental theorem for the symplectic group $R$ is
a quotient of $\Sym_k(\wedge^2 H)$ so that dually $\Spec R\hookrightarrow \wedge^2 H^\vee\subset \Hom_k(H,H^\vee)$.
The second fundamental theorem for the symplectic group yields
\[
\Spec R=\{\psi\mid \psi\in \Hom_k(H,H^\vee),\psi+\psi^\vee=0,\rk \psi\le r\}\,.
\]
so that $\Spec R\cong Y$ with $Y$ as introduced above. Below we identify $R$
with $k[Y]$.

We now discuss the Springer resolution $p:Z\r Y$ as well as the inclusion
$\Dscr(\Lambda)\hookrightarrow \Dscr(Z)$ announced in \eqref{ref-1.1-1}.
Let $ F=\Gr(r,H)$ be the Grassmannian of $r$-dimensional quotients
$H\twoheadrightarrow Q$ of $H$ and put
\[
Z=\{(\phi,Q)\mid Q\in F,\phi\in \Hom_k(Q,Q^\vee),\phi+\phi^\vee=0\}\,.
\]
The Springer resolution $p:Z\r Y\hookrightarrow  \Hom_k(H,H^\vee)$  of $Y$  sends 
$(\phi,Q)$ to the composition
\[
[H\twoheadrightarrow Q\xrightarrow{\phi} Q^\vee \hookrightarrow H^\vee]\in \Hom_k(H,H^\vee)\,.
\]
 Using again the fundamental theorems for
the symplectic group we have
\begin{equation}
\label{ref-1.4-5}
\Sym_k(Q\otimes_k V^\vee)^{\Sp(V)}\cong\Sym_k(\wedge^2 Q)
\end{equation}
(since $\dim Q=\dim V$, there are no relations on the righthand side).
For a partition~${{\chi}}$ with $l({{\chi}})\le r/2$ we put 
\begin{equation}
\label{eq:mq}
M_Q({{\chi}})=(\det Q)^{\otimes r-n}\otimes_k (S^{\langle {{\chi}}\rangle} V\otimes_k \Sym_k(Q\otimes_k V^\vee))^{\Sp(V)}
\end{equation}
where we consider $M_Q({{\chi}})$ as a $\GL(Q)$-equivariant $\Sym_k(\wedge^2 Q)$-module via \eqref{ref-1.4-5}.

Choose a specific $(H\twoheadrightarrow Q)\in F$. One has
$F=\GL(H)/P_Q$ where $P_Q$ is the parabolic subgroup of $\GL(H)$ that
stabilizes the kernel of $H\twoheadrightarrow Q$.  We regard
$\GL(Q)$-equivariant objects tacitly as $P_Q$-equivariant objects
through the canonical morphism $P_Q\twoheadrightarrow \GL(Q)$.  Taking the fiber in
$Q$ defines an equivalence between $\coh(\GL(H),Z)$ and
$\mod(P_Q,\Zscr_Q)$ where $\Zscr_Q:=\Sym_k(\wedge^2 Q)$, whose inverse
will be denoted by $\widetilde{?}$.
Put
\[
\Mscr_{Z}({{\chi}})=\widetilde{M_Q({{\chi}})}\in \coh(\GL(H),Z)\,.
\]
\begin{theorem} (see \S\ref{ref-5.2-29})
\label{ref-1.2-6}  Let $\mu,\lambda\in B_{r/2,n-r}$. 
\begin{enumerate}
\item
\label{ref-1-7}
We have for $i>0$.
\[
\Ext^i_Z(\Mscr_Z(\lambda),\Mscr_Z(\mu))=0.
\]
\item  There are isomorphisms as $R$-modules
\begin{equation}
\label{ref-1.5-8}
R\Gamma(Z,\Mscr_Z(\lambda))\cong M(\lambda)\,.
\end{equation}
\item \label{ref-3-9}
Applying $p_\ast$  induces an isomorphism
\begin{align}
\Hom_Z(\Mscr_Z(\lambda),\Mscr_Z(\mu))&\overset{p_\ast}{\cong} \Hom_Y(p_\ast\Mscr_Z(\lambda),p_\ast\Mscr_Z(\mu))\label{ref-1.6-10}\\
&\cong \Hom_R(\Gamma(Z,\Mscr_Z(\lambda)),\Gamma(Z,\Mscr_Z(\mu)))&&\text{($Y$ is affine)}\nonumber\\
&\cong \Hom_R(M(\lambda),M(\mu))&& \text{(by \eqref{ref-1.5-8})}\nonumber
\end{align}
\end{enumerate}
\end{theorem}
From this theorem it follows in particular that
\[
\Mscr_Z:=\bigoplus_{{{\chi}}\in B_{r/2,\lfloor n/2\rfloor-r/2}} \Mscr_Z({{\chi}})
\]
satisfies
\[
\Ext^i_Z(\Mscr_Z,\Mscr_Z)=
\begin{cases}
\Lambda&\text{if $i=0$}\\
0&\text{if $i>0$}
\end{cases}
\]
and we obtain the following more precise version of \eqref{ref-1.1-1}:
\begin{corollary} 
\label{ref-1.3-11} There is a full exact embedding
\[
-\Lotimes_{\Lambda} \Mscr_Z: \Dscr(\Lambda)\hookrightarrow \Dscr(Z)\,.
\]
\end{corollary}
\begin{remark} Put 
\[
M':=\bigoplus_{\chi\in B_{r/2,n-r}} M(\chi)
\]
and $\Gamma=\End_R(M')$. It follows from \cite[Thm 1.5.1]{SpenkoVdB}
(applied with $\Delta=\epsilon \bar{\Sigma}$ for a sufficiently small $\epsilon>0$)
 that $\gldim \Gamma<\infty$.  
See the computation in \S6 in loc.\ cit.. 
We have $\Lambda=e\Gamma e$ for a suitable idempotent $e$.
The fact that
$\gldim \Lambda<\infty$ implies that $\Gamma$ cannot be an NCCR
by  \cite[Ex.\ 4.34]{Wemyss1} (see also \cite[Remark 3.6]{SpenkoVdB}).
In the terminology 
of \cite{SpenkoVdB} $\Gamma$ is a (non-crepant) non-commutative
resolution of $R$. As in Corollary \ref{ref-1.3-11} we still have an embedding
$\Dscr(\Gamma)\subset \Dscr(Z)$.
\end{remark}
\section{Acknowledgement}
This paper owes a great deal to Sasha Kuznetsov who made the initial
conjecture that there should be an embedding
$\Dscr(\Lambda)\hookrightarrow \Dscr(Z)$ and at the same time also indicated a possible
proof. Furthermore, after carefully reading a preliminary version of this
manuscript, Kuznetsov suggested that the embedding we had constructed at that time might be coming from
a  ``splitting
functor'' \cite[\S3]{Kuznetsov3} $\Dscr(Z)\r \Dscr(X/\Sp(V))$. This observation turned out to be correct
and has allowed us to simplify our proofs and moreover to clarify our
statements. See \S\ref{ref-5-27}.

\medskip

After this paper was posted on the arXiv Steven Sam informed us that one of our auxiliary results 
concerning the $M_Q(\chi)$ introduced above can be generalized using more sophisticated machinery 
(see Remark \ref{sam2} below).  We are grateful for these very insightful comments.


\section{A $\GL(Q)$-equivariant free resolution of $M_Q(\lambda)$}
\label{ref-3-12}
In this section we discuss some of the properties of the $\Sym_k(\wedge^2 Q)$-modules ${M_{Q}}(\lambda)$ introduced in the introduction.
We basically restate some results from \cite{WeymanSam} in our current language. To do this it will be convenient to consider
\[
N_Q(\chi):= (S^{\langle {{\chi}}\rangle} V\otimes_k \Sym_k(Q\otimes_k V^\vee))^{\Sp(V)}
\]
so that $M_Q(\chi)=(\det Q)^{\otimes r-n}\otimes_k N_Q(\chi)$. Since $\det Q$ is one-dimensional, $M_Q(\lambda)$ and $N_Q(\lambda)$ have identical
properties.

The following fact will not be used although it seems interesting to know
\begin{lemma} ${N_Q}(\lambda)$ is a reflexive $\Sym_k(\wedge^2Q)$-module.
\end{lemma}
\begin{proof} This follows for example from the fact that $\Spec\Sym_k(Q\otimes_k V^\vee)\r \Spec\Sym_k(\wedge^2 Q)$ contracts 
no divisor.
\end{proof}
\label{ref-3-13}

Recall that a border strip is a connected skew Young diagram not
containing any $2\times 2$ square. The size of a border strip is the
number of boxes it contains. We follow \cite{WeymanSam} and associate
to some partitions $\lambda$ a partition $\tau_{r}(\lambda)$ and a
number $i_{r}(\lambda)$. The definition of $(\tau_r(\lambda),i_r(\lambda))$ is inductive. If $l(\lambda)\leq r/2$ then
$\tau_{r}(\lambda)=\lambda$, $i_r(\lambda)=0$. Suppose now that
$l(\lambda)>r/2$.  If there exists a non empty border strip $R_\lambda$
of size $2 l(\lambda)-r-2$ starting at the first box in the bottom
row of $\lambda$ such that $\lambda\setminus R_\lambda$ is a partition
then $\tau_r(\lambda):=\tau_r(\lambda\setminus R_\lambda)$, and
$i_r(\lambda):=c(R_\lambda)+i_r(\lambda\setminus R_\lambda)$, where
$c(R_\lambda)$ is the number of columns of $R_\lambda$.
Otherwise $\tau_r(\lambda)$ is undefined and $i_r(\lambda)=\infty$.

From \cite[Corollary 3.16]{WeymanSam} we
extract the following result (the role of $\Sym(\wedge^2 Q)$ is played by the ring $A$ in loc.\ cit. 
and  our $\Sym_k(Q\otimes_k V^\vee)$ is denoted by $B$).
\begin{proposition}
\label{ref-3.2-14}
  Assume $\chi$ is a partition with $l(\chi)\le r/2$. Then 
  ${N_Q}(\chi)$ has a $\GL(Q)$-equivariant free resolution as a $\Sym(\wedge^2
  Q)$-module which in homological degree $t\ge 0$ is the direct sum of 
$S^{\lambda}
  Q\otimes_k \Sym_k(\wedge^2 Q)$ for $\lambda$ satisfying $(\tau_r(\lambda),i_r(\lambda))=(\chi,t)$.
\end{proposition}
\begin{example}
\label{ref-3.3-15} Write $[\mu_1,\mu_2,\ldots]$ for $S^{\mu} Q\otimes_k  \Sym_k(\wedge^2 Q)$. Assume $r=4$. Then the above resolution of ${N_Q}(a,b)$ has the form
\[
0\r [a,b,1,1]\r [a,b]
\]
if $b\ge 1$. If $b=0$ then the resolution has only one term given by $[a]$.
\end{example}
\begin{example} 
\label{ref-3.4-16} Assume $r=6$. Now the resolution of ${N_Q}(a,b,c)$ is
\[
0\r [a,b,c,2,2,2] \r [a,b,c,2,1,1] \r [a,b,c,1,1]\r [a,b,c]
\]
if $c\ge 2$. If $c=1$ then we have
\[
0\r [a,b,1,1,1]\r [a,b,1]
\]
If $c=0,b\ge 1$ we get
\[
0 \r [a,b,1,1,1,1]\r [a,b]
\]
Finally for $c=b=0$ the resolution has again only a single term given
by $[a]$.
\end{example}
\begin{remark}
\label{sam2} 
For $\chi_{r/2}\ge r/2-1$ we give an explict description the resolution of $N_Q(\chi)$ (including the differentials) in Appendix \ref{ref-A-50}. 

Steven Sam informed us of an alternative (and more general) approach as follows. There is an action of $\mathfrak{so}(Q+Q^*)$ on $\Sym(Q\otimes_k V^*)$, which commutes with the $\Sp(V)$-action. Therefore $\mathfrak{so}(Q+Q^*)$ acts on $N_Q(\chi)$.  The resolution of $N_Q(\chi)$ in Proposition \ref{ref-3.2-14} can be upgraded to an $\mathfrak{so}(Q+Q^*)$-equivariant resolution, which is a BGG-resolution by parabolic Verma modules of the irreducible highest weight representation $N_Q(\chi)$.
This follows by \cite[Lemma 5.14, Theorem 5.15, Corollary 6.8]{EHP} since $N_Q(\chi)$ is  unitary \cite[Proposition 4.1]{ChengZhang}.

In this way, using \cite[Section 5.3]{EHP} and \cite[Proposition 3.7]{Lepowsky}, one may in fact give an explicit description of the resolution of $N_Q(\chi)$  also for general $\chi$. 
However an analogue of the uniqueness claim of Proposition \ref{prop:uniqueness} is apparently not yet available in the literature.
\end{remark}

Looking at the Examples \ref{ref-3.3-15}, \ref{ref-3.4-16} suggests the following easy consequence of Proposition   \ref{ref-3.2-14} which is crucial for what follows:
\begin{corollary}
\label{ref-3.6-19} The summands of the resolution of ${N_Q}(\chi)$ given in 
Proposition \ref{ref-3.2-14} are all of the form $S^\delta Q\otimes_k \Sym(\wedge^2 Q)$ with $\delta_1=\chi_1$. 
\end{corollary}
\begin{proof} 
Note first of all that $l(\delta)\le r$ (otherwise $S^\delta Q=0$).  
A border strip $R$ of size $\leq 2 l(\lambda)-r-2$ starting at the first box in the bottom row of 
a partition $\lambda$ with $r\ge l(\lambda)>r/2$ has at most $2 l(\lambda)-r-2$ rows. So if
we remove $R$ then the first $l(\lambda)-(2 l(\lambda)-r-2)=-l(\lambda)+r+2\ge 2$ rows of $\lambda$ are unaffected.
%

If $S^\delta Q\otimes_k \Sym(\wedge^2 Q)$, $\delta\neq \chi$, appears
in the resolution of ${N_Q}(\chi)$ then $\chi$ is by Proposition
\ref{ref-3.2-14} obtained from $\delta$ by a sequence of border strip removals
as in the previous paragraph. Thus $\delta_1=\chi_1$ (and also $\delta_2=\chi_2$).
\end{proof}
\section{The Springer resolution}
\label{ref-4-20}
Let $\sigma:V\r V^\vee$, $\sigma+\sigma^\vee=0$  be the isomorphism corresponding to the symplectic form on $V$. 
Consider the following diagram.
\begin{equation}
\label{ref-4.1-21}
\xymatrix{
{{E}}\ar[r]^{\tilde{p}}\ar[d]_{\tilde{q}}&X\ar[d]^q\\
Z\ar[r]_p\ar[d]_\pi&Y\ar@{^(->}[r]&\wedge^2 H^\vee\\
F
}
\end{equation}
where $X=\Hom(H,V)$ is as above and
\begin{equation}
\label{ref-4.2-22}
{{Y}}=\{\psi\in \Hom(H,H^\vee)\mid \psi+\psi^\vee=0,\rk \psi\le r\}\subset \wedge^2 H^\vee\,,
\end{equation}
\[
F=\Gr(r,H):=\{\text{$r$-dimensional quotients of $H$}\}\,,
\]
\begin{equation}
  \label{ref-4.3-23}
Z=\{(\phi,Q)\mid Q\in F,\phi\in \Hom(Q,Q^\vee),\phi+\phi^\vee=0\}\,,
\end{equation}
\begin{equation}
\label{ref-4.4-24}
{{E}}=\{(\epsilon,Q)\mid Q\in F, \epsilon\in \Hom(Q,V)\}\,.
\end{equation}
If $\theta:H\r V\in X$ then $q(\theta)\in {{Y}}$ is the composition
\[
q(\theta)=[H\xrightarrow{\theta}V\xrightarrow{\sigma}V^\vee
\xrightarrow{\theta^\vee} H^\vee]\,.
\]
If $(\phi,Q)\in Z$ then $p(\phi,Q)\in {{Y}}$ is the composition
\[
p(\phi,Q)=[H\twoheadrightarrow Q\xrightarrow{\phi} Q^\vee \hookrightarrow H^\vee]\,.
\]
The map $\pi:Z\r F$ is the projection $(\phi,Q)\mapsto Q$.  If $(\epsilon,Q)\in {{E}}$ then
$\tilde{p}(\epsilon,Q)$ is the composition
\[
[H\twoheadrightarrow Q\xrightarrow{\epsilon} V]
\]
and $\tilde{q}(\epsilon,Q)$ is $(\phi,Q)$ where $\phi$ is the composition
\[
[Q\xrightarrow{\epsilon} V\xrightarrow{\sigma} V^\vee \xrightarrow{\epsilon^\vee} Q^\vee]\,.
\]
In the diagram \eqref{ref-4.2-22}, $X$, $Z$, $E$ are smooth, $p$ is a resolution of singularities
and $\pi$ and $\pi\tilde{q}$ are vector bundles. 
The
coordinate ring of $X$ is $T=\Sym_k(H\otimes_k V^\vee)$. For the other schemes in
\eqref{ref-4.1-21} we have
\begin{align*}
Y&=\Spec T^{\Sp(V)}\\
Z&=\underline{\Spec}_F \Zscr\\
{{E}}&=\underline{\Spec}_F \Escr
\end{align*}
with $\Zscr$, $\Escr$ being the sheaves of $\Oscr_F$-algebras given by
\begin{equation}
\label{ref-4.5-25}
\begin{aligned}
\Zscr&=\Sym_F(\wedge^2\Qscr)\\
\Escr&=\Sym_F(\Qscr\otimes_k V^\vee)
\end{aligned}
\end{equation}
where $\Qscr$ is the tautological quotient bundle on $F$. From \eqref{ref-4.5-25} obtain
in particular
\begin{lemma} 
\label{ref-4.1-26} If $U\subset F$ is an affine open then $\pi^{-1}(U)$ 
and $(\pi\tilde{q})^{-1}(U)$ are affine and moreover $k[\pi^{-1}(U)]=
k[(\pi\tilde{q})^{-1}(U)]^{\Sp(V)}$.
\end{lemma}
Now let $Y_0$ be the open subscheme of $Y$ of those $\psi\in Y$
(see \eqref{ref-4.2-22}) which have rank exactly $r$ and put $X_0=q^{-1}(Y_0)$,
$Z_0=p^{-1}(Y_0)$, ${{E}}_0=\tilde{p}^{-1}(X_0)=\tilde{q}^{-1}(Z_0)$. Then it is easy to see that 
 $X_0\subset X=\Hom(H,V)$ is the open subscheme of those
$\theta:H\r V$ which are surjective and that
$Z_0\subset Z$ is the open subscheme of those  $(\phi,Q)\in Z$ (see
\eqref{ref-4.3-23}) where $\phi$ is an isomorphism. Finally ${{E}}_0$ is 
the open subscheme of ${{E}}$ of those $(\epsilon,Q)$ where $\epsilon$
is an isomorphism. The restricted morphisms $\tilde{q}_0:E_0\r Z_0$, $q_0:X_0\r Y_0$
are $\Sp(V)$-torsors and the restricted morphisms $\tilde{p}_0:E_0\r X_0$, $p_0:Z_0\r Y_0$
are isomorphisms.

\section{A splitting functor}
\label{ref-5-27}
\subsection{Preliminaries}
The idea of using splitting functors was suggested to us by Sasha Kuznetsov.
Recall that a (full) triangulated subcategory of a triangulated category is
 right admissible if the inclusion functor has a right adjoint.  
Following \cite[Def.\ 3.1]{Kuznetsov3} we say that a functor $\Phi:\Bscr\r \Ascr$ is right splitting
if $\ker \Phi$ is right admissible in~$\Bscr$,~$\Phi$ restricted to $(\ker \Phi)^\perp$
is fully faithful and finally $\im\Phi=\Phi(\ker\Phi^\perp)$ is right admissible in 
$\Ascr$. Left splitting functors are defined in a similar way.
We see that splitting functors are categorical versions of partial isometries between Hilbert space. 

According to \cite[Lem.\ 3.2, Cor.\ 3.4]{Kuznetsov3} a right splitting functor $\Phi$
has a right adjoint~$\Phi^!$ which is a left splitting functor. According to \cite[Thm 3.3(3r)]{Kuznetsov3} if
$\Phi$ is right splitting then~$\Phi$ and $\Phi^!$ induce inverse equivalences between
$\im \Phi\subset \Ascr$ and $\im \Phi^!\subset \Bscr$.
Below we will use the following criterion to verify that a certain functor is splitting.
\begin{lemmas} 
  \label{ref-5.1.1-28} Assume that $\Phi:\Bscr\r \Ascr$ is an exact
  functor between triangulated categories. Assume that $\Phi$ has a
  right adjoint $\Phi^!$ such that the composition of the counit map
  $\Phi \Phi^!\r \Id_{\Ascr}$ with $\Phi$ yields a natural isomorphism
 $\Phi \Phi^!\Phi\r \Phi$.
Then $\Phi$ is a right splitting functor.
\end{lemmas}
\begin{proof} 
This is equivalent to the criterion \cite[Thm 3.3(4r)]{Kuznetsov3}. In the latter case
we start from the unit map  $\Id_{\Ascr} \r \Phi\Phi^! $ and
we require that the resulting $\Phi\r \Phi \Phi^!\Phi$ is an isomorphism.
As the composition $\Phi\r \Phi \Phi^!\Phi
\r \Phi$ is the identity, it  follows that if one of these maps is an isomorphism then so
is the other.
\end{proof}
\subsection{The functor}
\label{ref-5.2-29}
  The diagram
\eqref{ref-4.1-21}
may be transformed into a diagram of quotient stacks 
\[
\xymatrix{
{{E/\Sp(V)}}\ar[r]^{\tilde{p}_s}\ar[d]_{\tilde{q}_s}&X/\Sp(V)\ar[d]^{q_s}\\
Z\ar[r]_p\ar[d]_\pi&Y\ar@{^(->}[r]&\wedge^2 H^\vee\\
F
}
\]
which is compatible with the natural maps $E\r E/\Sp(V)$, $X\r
X/\Sp(V)$.  This means in particular that $L\tilde{q}^\ast_s$, $Lq^\ast_s$,
$R\tilde{p}_{s,\ast}$, $L\tilde{p}_{s}^\ast$, $\tilde{p}^!_{s,\ast}$ may be computed like
their non-stacky counterparts. We will use this without further comment. 

We define the functor $\Phi$ as the composition
\[
\Phi:\Dscr(Z)\xrightarrow{L\tilde{q}^\ast_s} 
\Dscr(E/\Sp(V))
\xrightarrow{R\tilde{p}_{s,\ast}} \Dscr(X/\Sp(V))
\]
The functor $\Phi$ has a right adjoint $\Phi^!$ given by the composition
\[
\Phi^!:\Dscr(X/\Sp(V))\xrightarrow{\tilde{p}^!_s}  \Dscr(E/\Sp(V))\xrightarrow{R\tilde{q}_{s\ast}} \Dscr(Z)
\]
where
$
\tilde{p}_s^!=\omega_{E/X}\otimes_E L\tilde{p}_s^\ast(-)
$
and $\tilde{q}_{s\ast}$ is given by taking $\Sp(V)$-invariants. From Lemma \ref{ref-4.1-26} and the
fact that $\Sp(V)$ is reductive it follows that $\tilde{q}_{s\ast}$ is an exact functor.
\begin{theorems} 
\label{ref-5.2.1-30}
\begin{enumerate}
\item $\Phi$ is a right splitting functor. \label{ref-1-31}
\item $\im \Phi$ is the smallest triangulated subcategory of $\Dscr(X/\Sp(V))$ containing
$S^{\langle \lambda\rangle} V\otimes_k \Oscr_X$ for $\lambda\in B_{r/2,n-r}$. \label{ref-2-32}
\item $\im \Phi^!$ is the smallest triangulated subcategory of $\Dscr(Z)$ containing
$\Mscr_Z(\lambda)$ for $\lambda\in B_{r/2,n-r}$. \label{ref-3-33}
\item For $\lambda\in B_{r,n-r}$ we have \label{ref-4-34}
\[
\Phi(\pi^\ast((\det \Qscr)^{\otimes r-n} \otimes_F S^\lambda \Qscr))\cong S^\lambda V\otimes_k \Oscr_X\,.
\]
\item For $\lambda\in B_{r/2,n-r}$ we have \label{ref-5-35}
\[
\Phi(\Mscr_Z(\lambda))\cong S^{\langle \lambda\rangle} V\otimes_k\Oscr_X\,.
\]
\item For $\lambda\in B_{r/2,n-r}$ we have \label{ref-6-36}
\[
\Phi^!(S^{\langle \lambda\rangle} V\otimes_k\Oscr_X)\cong \Mscr_Z(\lambda).
\]
\end{enumerate}
\end{theorems}
The proof is based on a series of lemmas. Most arguments are quite standard. See \cite{VdB100,
WeymanBook}.
\begin{lemmas}\label{ref-5.2.2} \begin{enumerate}
\item We have 
\begin{equation}
\label{ref-5.1-37}
\omega_{E/X}= (\pi\tilde{q})^\ast (\det \Qscr)^{\otimes r-n}
\end{equation}
as $\GL(H)\times\Sp(V)$-equivariant coherent sheaves.
\item Moreover
\begin{equation}
\label{ref-5.2-38}
R\tilde{p}_{s,\ast} \omega_{E/X}=\Oscr_X.
\end{equation}
\end{enumerate}
\end{lemmas}
\begin{proof} 
\begin{enumerate}
\item
For clarity we will work $\GL(H)\times \GL(V)$-equivariantly.
Using the identification $E=\underline{\Spec} \Escr$ 
with $\Escr=\Sym_F(\Qscr\otimes_k V^\vee)$ (see \eqref{ref-4.5-25})
we find
that~$\omega_E$ corresponds to the sheaf of graded $\Escr$-modules 
given by
\[
\omega_{\Escr}=\omega_{F}\otimes_F\det (\Qscr\otimes_k V^\vee)\otimes_F\Escr\,.
\]
From the fact that $\Omega_F=\uHom_F(\Qscr,\Rscr)$ where $\Rscr=\ker(H\otimes_k\Oscr_F\r \Qscr)$
one computes
\[
\omega_F=(\det H)^{\otimes r}\otimes_F(\det \Qscr)^{\otimes -n}\,.
\]
We also have
\[
\det(\Qscr\otimes_k V^\vee)=(\det Q)^{\otimes r}\otimes_k (\det V)^{\otimes -r}
\]
so that ultimately we get 
\[
\omega_\Escr=(\det H)^{\otimes r}\otimes_k  (\det V)^{\otimes -r}\otimes_k (\det \Qscr)^{\otimes r-n}\otimes_F \Escr\,
\]
and hence
\[
\omega_E=(\det H)^{\otimes r}\otimes_k  (\det V)^{\otimes -r}\otimes_k(\pi\tilde{q})^\ast  (\det \Qscr)^{\otimes r-n}.
\]
One also has
$\omega_X=(\det H)^{\otimes r}\otimes_k (\det V)^{\otimes -n}\otimes_k \Oscr_X$ which yields 
\[
\omega_{E/X}=(\det V)^{\otimes n-r}\otimes_k(\pi\tilde{q})^\ast (\det \Qscr)^{\otimes r-n}.
\]
It now suffices to note that $\det V$ is a trivial $\Sp(V)$-representation.
\item It is easy to show this directly from \eqref{ref-5.1-37}
but one may also argue that~$X$, being smooth,
has rational singularities and hence $R\tilde{p}_{s,\ast}(\omega_E)=\omega_X$. Tensoring
with $\omega_X^{-1}$ yields the desired result.\qed
\end{enumerate}
\def\qed{}\end{proof}
On $E$ there is a tautological map
\[
\epsilon:(\pi\tilde{q})^\ast(\Qscr)\r V\otimes_k \Oscr_E
\]
whose fiber in a point $(\epsilon,Q)\in E$ is simply $\epsilon:Q\r V$. From this description it is clear
that $\epsilon{|}E_0$ is an isomorphism.
\begin{lemmas} 
\label{ref-5.2.3-39} Assume $\lambda\in B_{r,n-r}$. The map $S^\lambda \epsilon$
becomes an isomorphism after applying the functor $R\tilde{p}_\ast (\omega_{E/X}\otimes_E-)$.
\end{lemmas}
\begin{proof} 
By \eqref{ref-5.2-38} we have
\begin{equation}
\label{ref-5.3-40}
R\tilde{p}_\ast(\omega_{E/X}\otimes_E(S^\lambda V\otimes_k \Oscr_E))=
S^\lambda V\otimes_k  R\tilde{p}_\ast(\omega_{E/X})=
S^\lambda V\otimes_k  \Oscr_X.
\end{equation}
When viewed as $\Sym_k(H\otimes_k V^\vee)$-module 
$R^i\tilde{p}_\ast(\omega_{E/X}\otimes_E S^\lambda((\pi\tilde{q}_s)^\ast(\Qscr)))$ is given by
\begin{equation}
\label{ref-5.4-41}
H^i(F,S^\lambda \Qscr \otimes_F (\det \Qscr)^{\otimes r-n}\otimes_F \Escr)
\end{equation}
(using \eqref{ref-5.1-37}).  It follows from \cite[Prop.\ 1.4]{BLV1000} that
\eqref{ref-5.4-41} is zero for $i>0$. So
$R^i\tilde{p}_\ast(\omega_{E/X}\otimes_E
S^\lambda((\pi\tilde{q}_s)^\ast(\Qscr)))=0$ for $i>0$. We now consider $i=0$.  We claim that $\tilde{p}_\ast(\omega_{E/X}\otimes_E
S^\lambda((\pi\tilde{q}_s)^\ast(\Qscr))$ is maximal Cohen-Macaulay. 
  To this end we have to show that
$\RHom_X(\tilde{p}_\ast(\omega_{E/X}\otimes_E
S^\lambda((\pi\tilde{q}_s)^\ast(\Qscr)),\Oscr_X)$ has no higher
cohomology or equivalently $\Ext^i_E(
S^\lambda((\pi\tilde{q})^\ast(\Qscr)),\Oscr_E)=0$ for $i>0$. In other
words we should have
\[
H^i(F,(S^\lambda \Qscr)^\vee \otimes_F \Escr)=0
\]
for $i>0$. This follows again from \cite[Prop.\ 1.4]{BLV1000}. 

Combining this with \eqref{ref-5.3-40} we see that
$R\tilde{p}_\ast(\omega_{E/X}\otimes_X S^\lambda\epsilon)$ is a map
between maximal Cohen-Macaulay $\Oscr_X$-modules. Since this map is an
isomorphism on $X_0$ and $\codim (X-X_0)\ge 2$ we conclude that
$R\tilde{p}_\ast(\omega_{E/X}\otimes_X S^\lambda\epsilon)$ is indeed an isomorphism.
\end{proof}
Put $\Nscr_Z(\lambda):=\widetilde{N_Q(\lambda)}$ where the notation $\tilde{?}$
was introduced in the introduction and $N_Q(\lambda)$ was introduced in
\S\ref{ref-3-12}.
From Lemma \ref{ref-4.1-26} 
we deduce
\begin{equation}
\label{ref-5.5-42}
\Nscr_Z(\lambda)\cong R\tilde{q}_{s,\ast}(S^{\langle \lambda\rangle}V\otimes_k \Oscr_E)
\end{equation}
so that by adjunction we get a map
\begin{equation}
\label{ref-5.6-43}
L\tilde{q}^\ast_s\Nscr_Z(\lambda)\r S^{\langle \lambda\rangle}V\otimes_k \Oscr_E.
\end{equation}
\begin{lemmas} 
\label{ref-5.2.4-44} Assume $\lambda\in B_{r/2,n-r}$. The map \eqref{ref-5.6-43} becomes an isomorphism after applying
the functor $R\tilde{p}_{s,\ast}(\omega_{E/X}\otimes_E-)$.
\end{lemmas}
\begin{proof}  Note that \eqref{ref-5.6-43}
is an isomorphism on $E_0$ since $E_0\r Z_0$ is an $\Sp(V)$-torsor
and so $L\tilde{q}^\ast_s$
and $R\tilde{q}_{s,\ast}$ define inverse equivalences between $\Dscr(E_0/\Sp(V))$
and $\Dscr(Z_0)$.

By Corollary \ref{ref-3.6-19} we have a $\GL(H)$-equivariant resolution
\[
\cdots \r P_1(\pi^\ast\Qscr)\r P_0(\pi^\ast\Qscr)\r \Nscr_Z(\lambda)\r 0
\]
where the $P_i$ are polynomial functors which are finite sums of Schur functors $S^\chi$ with $\chi\in B_{r,n-r}$. 
 It follows that the cone of \eqref{ref-5.6-43} is described by a $\GL(H)\times \Sp(V)$-equivariant complex of the
form
\begin{equation}
\label{ref-5.7-45}
\cdots \r P_1((\pi\tilde{q})^\ast\Qscr)\r P_0((\pi\tilde{q})^\ast\Qscr)\r  S^{\langle \lambda\rangle}V\otimes_k \Oscr_E\r 0
\end{equation}
and moreover this complex is exact when restricted to $E_0$. Using Lemma \ref{ref-5.2.3-39} 
and \eqref{ref-5.2-38} applying $R\tilde{p}_{\ast}(\omega_{E/X}\otimes_X-)$ to \eqref{ref-5.7-45} yields a $\GL(H)\times \Sp(V)$-equivariant complex on $X$
\begin{equation}
\label{ref-5.8-46}
\cdots \r P_1(V)\otimes_k \Oscr_X \r P_0(V)\otimes_k \Oscr_X\r  S^{\langle \lambda\rangle}V\otimes_k \Oscr_X\r 0
\end{equation}
This complex is exact on $X_0$ (since $X_0\cong E_0$) but we must prove it is exact on $X$.
The morphisms in \eqref{ref-5.8-46} are determined by $\GL(H)\times\Sp(V)$-equivariant maps
\begin{align*}
P_{i+1}(V)&\r P_i(V)\otimes_k \Sym_k(H\otimes_k V^\vee)\\
P_{0}(V)&\r S^{\langle\lambda\rangle}(V)\otimes_k \Sym_k(H\otimes_k V^\vee)
\end{align*}
which by $\GL(H)$-equivariance must necessarily be obtained from $\Sp(V)$-equivariant maps
\begin{align*}
P_{i+1}(V)&\r P_i(V)\\
P_{0}(V)&\r S^{\langle\lambda\rangle}(V)
\end{align*}
We conclude that \eqref{ref-5.8-46} is of the form
\begin{equation}
\label{ref-5.9-47}
(\cdots \r P_2(V)\r P_1(V)\r P_0(V)\r S^{\langle\lambda\rangle} V\r 0)\otimes_k \Oscr_X
\end{equation}
in a way which is compatible with $\GL(H)\times \Sp(V)$-actions. Restricting to $X_0$  we see that
\[
\cdots \r P_2(V)\r P_1(V)\r P_0(V)\r S^{\langle\lambda\rangle} V\r 0
\]
must be exact. But then \eqref{ref-5.9-47} is also exact and hence so is \eqref{ref-5.8-46}.
\end{proof}
\begin{lemmas} 
\label{ref-5.2.5-48} Let $\lambda\in B_{r/2,n-r}$.
The counit map
\[
\Phi\Phi^!(S^{\langle \lambda\rangle}V\otimes_k \Oscr_X)
\r S^{\langle \lambda\rangle}V\otimes_k \Oscr_X
\]
is an isomorphism.
\end{lemmas}
\begin{proof} We have
\[
\tilde{p}_{s}^!(S^{\langle \lambda\rangle}V\otimes_k \Oscr_X)
=S^{\langle \lambda\rangle}V\otimes_k \omega_{E/X}.
\]
Hence we have to show that the counit map
\[
L\tilde{q}_s^\ast R{\tilde{q}}_{s,\ast} (S^{\langle \lambda\rangle}V\otimes_k \omega_{E/X})
\r S^{\langle \lambda\rangle}V\otimes_k \omega_{E/X}
\]
becomes an isomorphism after applying $R\tilde{p}_{s,\ast}$. 

Using \eqref{ref-5.1-37} we see that it is sufficient to prove that
\[
L\tilde{q}_s^\ast R{\tilde{q}}_{s,\ast} (S^{\langle \lambda\rangle}V\otimes_k \Oscr_E)
\r S^{\langle \lambda\rangle}V\otimes_k \Oscr_E
\]
becomes an isomorphism after applying $R\tilde{p}_{s,\ast}(\omega_{E/X}\otimes_X-)$.
This is precisely Lemma \ref{ref-5.2.4-44}.
\end{proof}
\begin{proof}[Proof of Theorem \ref{ref-5.2.1-30}]
\begin{enumerate}
\item[\eqref{ref-6-36}]
We have by \eqref{ref-5.1-37} and \eqref{ref-5.5-42}
\begin{align*}
\Phi^!(S^{\langle \lambda\rangle} V\otimes_k\Oscr_X)&=R\tilde{q}_{s,\ast}(\omega_{E/X}\otimes_E (S^{\langle \lambda\rangle} V\otimes_k\Oscr_E))\\
&=\pi^\ast(\det \Qscr)^{\otimes r-n} \otimes_Z \Nscr_Z(\lambda)\\
&=\Mscr_Z(\lambda).
\end{align*}
\item[\eqref{ref-5-35}]
Using \eqref{ref-5.1-37}\eqref{ref-5.2-38} and Lemma \ref{ref-5.2.4-44} we have
\begin{align*}
\Phi(\Mscr_Z(\lambda))&=R\tilde{p}_{s,\ast}(\omega_{E/X} \otimes_E L\tilde{q}_{s}^\ast \Nscr_Z(\lambda))\\
&=S^{\langle\lambda\rangle} V\otimes_k R\tilde{p}_{s,\ast}\omega_{E/X}\\
&=S^{\langle\lambda\rangle} V\otimes_k \Oscr_X.
\end{align*}
\item[\eqref{ref-4-34}] Using \eqref{ref-5.1-37}\eqref{ref-5.2-38} and Lemma \ref{ref-5.2.3-39} we have 
\begin{align*}
\Phi(\pi^\ast((\det \Qscr)^{\otimes r-n} \otimes_F S^\lambda \Qscr))&=R\tilde{p}_{s,\ast}(\omega_{E/X}\otimes_E L(\pi\tilde{q}_s)^\ast(S^\lambda \Qscr))\\
&=S^\lambda V\otimes_k R\tilde{p}_{s,\ast}(\omega_{E/X})\\
&=S^\lambda V\otimes_k \Oscr_X.
\end{align*}
\item[\eqref{ref-1-31}] We use Lemma \ref{ref-5.1.1-28}. So we have to prove that the counit map $\Phi\Phi^{!}(A)\r A$ is an isomorphism
for every object of the form $A=\Phi(B)$ with $B\in \Dscr(Z)$. It is clearly sufficient to check this for $B$ running through a set of
generators of $\Dscr(Z)$. The sheaves $(\det \Qscr)^{\otimes r-n}\otimes_F S^\lambda\Qscr$ 
for $\lambda\in B_{r,n-r}$ generate $\Dscr(F)$ \cite{Kapranov3}. Hence since $Z\r F$
is affine it follows that the sheaves $\pi^\ast((\det \Qscr)^{\otimes r-n}\otimes_F S^\lambda\Qscr)$  generate
$\Dscr(Z)$. By \eqref{ref-4-34} we have $\Phi(\pi^\ast((\det \Qscr)^{\otimes r-n} \otimes_F S^\lambda \Qscr))
\cong S^\lambda V\otimes_k \Oscr_X$ and $S^\lambda V$ is a sum of $S^{\langle\mu\rangle}V$ with $\mu_1\le\lambda_1$,
for example by careful inspection of the formula \cite[\S2.4.2]{HTW}. It now suffices to invoke Lemma \ref{ref-5.2.5-48}
(or, with a bit of handwaving, \eqref{ref-5-35}\eqref{ref-6-36}).
\item[\eqref{ref-2-32}] This has been proved as part of \eqref{ref-1-31}.
\item[\eqref{ref-3-33}] By \cite[Thm 3.3(3r)]{Kuznetsov3} it follows that $\im \Phi^!=\Phi^!(\im \Phi)$. It now suffices to invoke \eqref{ref-2-32}\eqref{ref-6-36}.\qed
\end{enumerate}
\def\qed{}\end{proof}
\begin{proof}[Proof of  Theorem \ref{ref-1.2-6}]
\begin{enumerate}
\item Since by Theorem \ref{ref-5.2.1-30}(6) $\Mscr_Z(\lambda),\Mscr_Z(\mu)\in \im \Phi^!$ we have by Theorem \ref{ref-5.2.1-30}(5)
\begin{align*}
\Ext^i_Z(\Mscr_Z(\lambda),\Mscr_Z(\mu))&=\Ext^i_{X/\Sp(V)}(\Phi(\Mscr_Z(\lambda)),\Phi(\Mscr_Z(\mu)))\\
&=\Ext^i_{X/\Sp(V)}(S^{\langle \lambda\rangle}V\otimes_k \Oscr_X,S^{\langle \mu\rangle}V\otimes_k \Oscr_X)
\end{align*}
which is zero for $i>0$ (since $\Sp(V)$ is reductive). Note that we also find
\begin{equation}
\label{ref-5.10-49}
\begin{aligned}
\Hom_Z(\Mscr_Z(\lambda),\Mscr_Z(\mu))&= \Hom_X(S^{\langle \lambda\rangle}V\otimes_k \Oscr_X,S^{\langle \mu\rangle}V\otimes_k \Oscr_X)^{\Sp(V)}\\
&\cong\Hom_R(M(\lambda),M(\mu))
\end{aligned}
\end{equation}
by \cite[Lemma 4.1.3]{SpenkoVdB}. 
\item We have by \eqref{ref-5.1-37}\eqref{ref-5.2-38}
\begin{align*}
Rp_\ast \Mscr_Z(\lambda)&=Rp_\ast R\tilde{q}_{s,\ast}(\omega_{E/X}\otimes_k S^{\langle \lambda\rangle }V)\\
&=Rq_{s,\ast} R\tilde{p}_{s,\ast}(\omega_{E/X}\otimes_k S^{\langle \lambda\rangle }V)\\
&=Rq_{s,\ast}( S^{\langle \lambda\rangle }V\otimes_k R\tilde{p}_{s,\ast}(\omega_{E/X}))\\
&=Rq_{s,\ast} (S^{\langle \lambda\rangle }V\otimes_k \Oscr_X)\\
&=(S^{\langle \lambda\rangle }V\otimes_k \Oscr_X)^{\Sp(V)}
\end{align*}
Taking global sections yields what we want. 
\item By \eqref{ref-5.10-49} and \eqref{ref-1.5-8} both sides of \eqref{ref-1.6-10} are reflexive $R$-modules.
Since~$p_\ast$ induces an isomorphism on $Y_0$ between both sides of \eqref{ref-1.6-10} (viewed as sheaves on $Y$)
and $\codim(Y-Y_0)\ge 2$ \eqref{ref-1.6-10} must be an isomorphism. \qed
\end{enumerate}
\def\qed{}\end{proof}

\section{Symmetric matrices}\label{symsec}
In this section we present modification needed to treat determinantal varieties of symmetric matrices. 

We keep the same notation as in the introduction, but now we equip $V$ with a symmetric bilinear form 
so that $r=\dim V$ does not need to be even, $Y$ is the variety of $n\times n$ symmetric matrices of rank 
$\leq r$, $G=O(V)$, 
while 
 $X=\Hom(H,V)$, $T=\Sym_k(H\otimes V^\vee)$ remain the same, put $R=T^{O(V)}$. By the fundamental theorems for the orthogonal
 group we have $Y\cong \Spec R$.

 If $\chi$ is a partition with $\chi^t_1+\chi^t_2\leq r$, where
 $\chi^t$ denotes the transpose partition, we write $S^{[\chi]}V$
 for the corresponding irreducible representation of $O(V)$ (see
 \cite[\S 19.5]{FH}), and call such a partition admissible. By
 $\chi^\sigma$ we denote the conjugate partition of $\chi$; i.e.,
 $(\chi^\sigma)^t_1=r-\chi^t_1$, $(\chi^\sigma)^t_k=\chi^t_k$ for $k>1$. Note that either $l(\chi)\le r/2$ or 
$l(\chi^\sigma)\le r/2$.  We have $S^{[\lambda^\sigma]}V=\det V\otimes_k S^{[\lambda]} V$ \cite[\S6.6, Lemma 2]{Procesi3}.

In \cite{SpenkoVdB} a non-commutative resolution of $R$ has been constructed, which is crepant in case $n$ and $r$ have opposite parity.
Let $B_{k,l}^{a}$ denote the set admissible partitions in $B_{k,l}$. 
We put 
\begin{equation}
\label{ref-6-1}
M=\bigoplus_{{{\chi}}\in B_{r,\lfloor (n-r)/2\rfloor +1 }^{a}} M({{\chi}}),
\end{equation}
where $M(\chi)=(S^{[\chi]} V\otimes_k T)^{O(V)}$
and write $\Lambda=\End_R(M)$.

\begin{theorem}
One has $\gldim\Lambda<\infty$. 
$\Lambda$ is a non-commutative crepant resolution of $R$ 
if $n$ and $r$ have opposite parity.\footnote{In case $n$, $r$ have the same parity then  there is a \emph{twisted} non-commutative crepant resolution. We do not consider such resolutions in this paper.}
\end{theorem}


In the symmetric case we also have an analogous Springer resolution where we adapt the definitions in the obvious way. The fundamental theorems for the orthogonal group yield $\Sym_k(Q\otimes V)^{O(V)}\cong \Sym_k(\Sym^2(Q))$. 
We only slightly change the definition of $M_Q(\chi)$, now 
\[
M_Q({{\chi}})=\det(V)^{\rn}\otimes(\det Q)^{\otimes r-n}\otimes_k (S^{[ {{\chi}}]} V\otimes_k \Sym_k(Q\otimes_k V^\vee))^{O(V)},
\]
where $\rn=0$ (resp. $\rn=1$) if $r$ and $n$ have the same (resp. opposite) parity.  
As in the skew symmetric case 
$\Mscr_{Z}({{\chi}})=\widetilde{M_Q({{\chi}})}\in \coh(\GL(H),Z)$. 

To give an analogue of Proposition \ref{ref-3.2-14} 
we need to adapt the definitions of $\tau_r(\lambda)$, $i_r(\lambda)$ 
following \cite[\S 4.4]{WeymanSam}. 
The differences (denoted by D1, D2, D3 in loc.\ cit.) are that  we remove border strips $R_\lambda$ of size $2 l(\lambda)-r$ instead of $2l(\lambda)-r-2$ and in the definition of $i_r(\lambda)$ we use $c(R_\lambda)-1$ instead of $c(R_\lambda)$. Finally 
if the total number of border strips removed is odd, then we replace the 
end result $\mu$ with $\mu^\sigma$. 

With these modifications and replacing $B_{r/2,n-r}$ by $B_{r,n-r}^{a}$ Proposition \ref{ref-3.2-14} remains true also in the symmetric case by \cite[Corollary 4.23]{WeymanSam} in the case $r$ is odd, and by 
\cite[(4.2), Theorem 4.4]{WeymanSam} in the case $r$ is even. 
Also Corollary \ref{ref-3.6-19} remains valid. 
In its proof we only need to additionally note that one can also remove a border strip of size 
$l(\lambda)$ (which affects the first row) but this can only happen in the case $\lambda=(1^r)$
and in this case, since the number of borders strips removed is odd, $\tau_r(\lambda)=(0)^\sigma=\lambda$. In particular, $\tau_r(\lambda)_1=\lambda_1$ still holds.

We now present modifications needed in statements of other results. 
\begin{itemize}
\item
In Theorem \ref{ref-1.2-6} we replace $B_{r/2,n-r}$ by  $B_{r,n-r}^{a}$. 
\item
In Theorem \ref{ref-5.2.1-30} we replace $S^{\langle \lambda\rangle}V$ by 
$S^{[ {{\lambda}}]} V$, and $B_{r/2,n-r}$ by $B_{r,n-r}^{a}$. Item (4) needs to be modified as
\[
\Phi(\pi^\ast((\det \Qscr)^{\otimes r-n} \otimes_F S^\lambda \Qscr))\cong S^\lambda V\otimes_k (\det V)^{\rn}\otimes_k \Oscr_X\,.
\]
\item
In Lemma \ref{ref-5.2.2} we have 
\[
\omega_{E/X}=(\det V)^{\rn}\otimes (\pi\tilde{q})^\ast (\det \Qscr)^{\otimes r-n}
\]
as $\GL(H)\times O(V)$-equivariant coherent sheaves.
\end{itemize}

One can easily check that the proofs obtained in the skew symmetric case also apply almost verbatim  in the symmetric case.
\appendix
\section{More on the resolution of ${N_Q}(\chi)$ in the symplectic case}
\label{ref-A-50}

We refer to Remark \ref{sam2} for an alternative approach, suggested to us by Steven Sam, towards
the results in this Appendix. We believe that our elementary arguments
are still 
of independent interest.

\medskip

Recall that a partition has Frobenius
coordinates $(a_1,\ldots,a_u;b_1,\ldots,b_u)$, $a_1>\cdots>a_u\ge 1$, $b_1>\cdots>b_u\ge 1$ if for all $i$ the box
$(i,i)$ has arm length $a_i-1$  and leg length $b_i-1$. Let ${Q_{-1}}(m)$ be the set of partitions $\chi$ with
$|\chi|=m$ whose Frobenius coordinates are of the form $(a_1,\ldots,a_u{{;}}
a_1{+}1,\ldots,a_u{+}1)$.

For partitions $\delta,\chi$ such that $l(\delta)$, $l(\chi)\le r/2$ put
$(\delta|\chi):=(\delta_1,\ldots,\delta_{r/2},\chi_1,\ldots,\chi_{r/2})$
with the latter being viewed as a weight for $\GL(Q)$.  
For $\alpha\in {Q_{-1}}(2k)$, $\beta\in {Q_{-1}}(2(k-1))$, $l(\alpha),l(\beta)\le r/2$ we put
$\beta\subset_2 \alpha$ if $\beta\subset \alpha$ and $\alpha/\beta$ does not consist
of two boxes next to each other.

For $\chi$ a partition with $l(\chi)\le r/2$ and $\chi_{r/2}\ge r/2-1$ put
\[
S_{\chi,k}=\{(\chi|\mu)\mid \mu\in {Q_{-1}}(2k), l(\mu)\le r/2\}\,.
\]
Note 
that if $\mu\in {Q_{-1}}(2k)$ and $l(\mu)\le r/2$ then $\mu_1\le r/2-1$. Hence all elements of $S_{\chi,k}$ are dominant.
For $\pi=(\chi|\alpha)\in S_{\chi,k}$, $\tau=(\chi|\beta)\in
S_{\chi,k-1}$ put $\tau\subset_2\pi$ if $\beta\subset_2\alpha$. If
$\tau\subset_2\pi$ then by the Pieri rule $S^\tau Q$ is a summand with
multiplicity one of $\wedge^2 Q\otimes_k S^\pi Q$. 
 We call any
non-zero $\GL(Q)$-equivariant map
\[
\phi_{\pi,\tau}:S^\pi Q\r \wedge^2 Q\otimes_k S^\tau Q
\]
a Pieri map.
Needless to say that  a Pieri map is only determined up to a non-zero scalar.
By analogy of \cite[\S7]{VdB100} we call a collection
of Pieri-maps $\phi_{\pi,\tau}$ such that $\tau\subset_2 \pi$  a Pieri system.
We say that two Pieri systems $\phi_{\pi,\tau}$, $\phi'_{\pi,\tau}$ are equivalent if there exist
non-zero scalars $(c_\sigma)_\sigma$ such that
\[
\phi'_{\pi,\tau}=\frac{c_\tau}{c_\pi}\phi_{\pi,\tau}\,
\]
for all $\pi,\tau$.
We will now make Proposition \ref{ref-3.2-14} more explicit
for partitions with $\chi_{r/2}\ge r/2-1$.
\begin{proposition}
\label{prop:uniqueness}
  Assume $\chi$ is a partition with $l(\chi)\le r/2$ and $\chi_{r/2}\ge r/2-1$. Then 
  $N_{Q}(\chi)$ has a $\GL(Q)$-equivariant resolution $P_\bullet$ as a $\Sym(\wedge^2
  Q)$-module
such that
\[
P_k=\bigoplus_{\pi\in S_{\chi,k}}S^{\pi}
  Q\otimes_k \Sym_k(\wedge^2 Q)
\]
and such that the differential $P_k\r P_{k-1}$ is the sum of maps for $\tau\subset_2\pi$:
\begin{equation}
\label{ref-A.1-51}
S^\pi Q\otimes_k \operatorname{Sym}(\wedge^2 Q)\xrightarrow{\phi_{\pi,\tau}\otimes 1} S^\tau Q\otimes_k \wedge^2 Q\otimes_k \operatorname{Sym}(\wedge^2 Q)
\r  S^\tau Q\otimes_k \operatorname{Sym}(\wedge^2 Q)
\end{equation}
where the $(\phi_{\pi,\tau})_{\pi,\tau}$ are Pieri maps and the last map is obtained from the multiplication $\wedge^2 Q\otimes_k \operatorname{Sym}(\wedge^2 Q)\r\operatorname{Sym}(\wedge^2 Q)$. Moreover every choice of Pieri maps such that the compositions $P_k\r P_{k-1}\r P_{k-2}$ are zero yields isomorphic 
resolutions, and the isomorphism is given by scalar multiplication.
\end{proposition}
\begin{proof}
We will first discuss uniqueness up to scalar multiplication of maps in the resolutions.  The condition that \eqref{ref-A.1-51} forms a complex may be expressed as follows.
For  $\pi\in S_{\chi,k}$, $\sigma\in S_{\chi,k-2}$ put
\begin{equation}
\label{ref-A.2-52}
\{(\tau_i)_i\in I\}:=\{\tau\in S_{\chi,k-1}\mid \sigma\subset_2 \tau\subset_2 \pi\}\,.
\end{equation}
Then \eqref{ref-A.1-51} forms a complex if and only if the compositions
\begin{equation}
\label{ref-A.3-53}
S^\pi Q\xrightarrow{(\phi_{\pi,\tau_i})_i} \bigoplus_i \wedge^2 Q\otimes S^{\tau_i} Q 
\xrightarrow{(1\otimes\phi_{\tau_i,\sigma})_i} \wedge^2 Q\otimes \wedge^2 Q\otimes S^{\sigma} Q\r S^2(\wedge^2 Q )\otimes S^{\sigma} Q 
\end{equation}
are zero. 
We must show that
any two Pieri-systems satisfying \eqref{ref-A.3-53} are equivalent.

Let $\alpha\in {Q_{-1}}(2k)$, $\beta\in {Q_{-1}}(2(k-1))$.
We may express the relation $\beta\subset_2 \alpha$ in
  terms of Frobenius coordinates. If $\alpha=(a_1,\ldots,a_u{;}a_1+1,\ldots,a_u+1)$
and $\beta=(b_1,\ldots,b_v{;}b_1+1,\ldots,b_v+1)$ then $\beta\subset_2\alpha$ if and only if $u=v$ and
$(a_1,\ldots,a_u)=(b_1,\ldots,b_t+1,\ldots,b_v)$ for some $t$, or else $u=v+1$ and $(a_1,\ldots,a_{u})=(b_1,\ldots,b_v,1)$.
From this it follows in particular that \eqref{ref-A.2-52} contains at most two elements. 

Like in the proof of \cite[Prop.\ 7.1(iv)]{VdB100} we can now build a contractible cubical complex $\PP$ with vertices $\cup_k S_{\chi,k}$
and edges the pairs $\tau\subset_2\pi$ such that if $\phi_{\pi,\tau}$, $\phi'_{\pi,\tau}$ are two Pieri-systems satisfying 
\eqref{ref-A.1-51} then $\phi'_{\pi,\tau}/\phi_{\pi,\tau}$ is a 1-cocycle for $\PP$. Since $\PP$ is contractible this 1-cocycle is a coboundary
which turns out to express exactly that $\phi'_{\pi,\tau}$ and $\phi_{\pi,\tau}$ are equivalent. 

\medskip

We now discuss the existence of $P_\bullet$. To this end we introduce some notation. 
Let $G$ be the Grassmannian of $r/2$ dimensional quotients of $Q$ and let 
 $\Pscr$, $\Sscr$ be respectively the universal quotient and
subbundle on~$G$. The resolution of ${N_Q}(\chi)$ constructed in \cite[Lemma 3.11, Lemma 3.12, Prop.\ 3.13]{WeymanSam}
(denoted by $M_\chi$ in loc.\ cit.) using the ``geometric method'' is now obtained by applying $\Gamma(G,-\otimes_G S^\chi\Pscr)$ to the Koszul complex
\[
\wedge^\bullet (\wedge^2 \Sscr) \otimes_k \Sym(\wedge^2 Q)
\]
obtained from the inclusion $\wedge^2\Sscr\subset \Oscr_G\otimes_k \wedge^2 Q$. So the resulting complex is 
\begin{equation}
\label{ref-A.4-54}
\Gamma(G,\wedge^\bullet(\wedge^2\Sscr)\otimes_G S^\chi \Pscr)\otimes_k \Sym(\wedge^2 Q)\,.
\end{equation}
Using the decomposition
\begin{equation}
\label{ref-3.2-18}
\wedge^{k}(\wedge^2\Sscr)\cong \bigoplus_{\mu\in {Q_{-1}}(2k)} S^\mu\Sscr\,
\end{equation}
we obtain from Lemma \ref{ref-A.2-56} below that the differential in \eqref{ref-A.4-54} is 
given by the composition
\begin{multline}
\label{ref-A.5-55}
\Gamma(G,S^\chi  \Pscr\otimes_G S^\alpha \Sscr)
\xrightarrow{\phi_{\alpha,\beta,\Sscr}}
\Gamma(G,S^\chi  \Pscr\otimes_G \wedge^2\Sscr \otimes_G S^\beta \Sscr)
\hookrightarrow\\
\Gamma(G,S^\chi  \Pscr\otimes_G (\wedge^2 Q\otimes_k S^\beta \Sscr))
=\Gamma(G,S^\chi  \Pscr\otimes_G S^\beta \Sscr) \otimes_k \wedge^2Q
\end{multline}
where $\phi_{\alpha,\beta,\Sscr}$ is a Pieri map. Now for each pair $(\chi,\alpha)\in S_{\chi,k}$ 
choose an isomorphism $\Gamma(G,S^\chi \Pscr\otimes_G S^\alpha \Sscr)\cong S^{(\chi|\alpha)} Q$.
Then \eqref{ref-A.5-55} becomes a $\GL(Q)$-equivariant morphism
\[
\phi_{\chi,\alpha,\beta}:S^{(\chi|\alpha)}Q
\r
S^{(\chi|\beta)}Q
 \otimes_k \wedge^2 Q\,.
\]
\begin{sublemma} If $\beta\subset_2\alpha$ then $\phi_{\chi,\alpha,\beta}$
is not zero and hence it is a Pieri map.
\end{sublemma}
\begin{proof}
In \eqref{ref-A.5-55} $\phi_{\alpha,\beta,\Sscr}$ is a monomorphism. So
it induces a monomorphism on global sections. The compositions of two monomorphisms is again
monomorphism. This can only be zero if its source is zero, which is not the case since $(\chi|\alpha)\in S_{\chi,k}$ is dominant.
\end{proof}
It follows that \eqref{ref-A.4-54} becomes a complex of the shape asserted in the statement of the proposition, finishing
the proof.
\end{proof}
A version for vector bundles of the following lemma was used.
\begin{lemma}
\label{ref-A.2-56} Let $R$ be a vector space of dimension $n$. Let $\alpha\in {Q_{-1}}(2k)$, $\beta\in {Q_{-1}}(2(k-1))$ with $\beta\subset_2\alpha$
and $l(\alpha)\le n$. Then following  composition is non-zero
\[
\phi_{\alpha,\beta}:S^\alpha R\hookrightarrow  \wedge^k(\wedge^2 R)\xrightarrow{\phi} \wedge^2R \otimes_k \wedge^{k-1} (\wedge^2 R) \twoheadrightarrow \wedge^2 R\otimes_l S^\beta R
\]
where the first and last map are obtained from the $\GL(R)$-equivariant decomposition $\wedge^k(\wedge^2 R)\cong\bigoplus_{\alpha\in {Q_{-1}}(2k)} S^\alpha R$,
$\wedge^{k-1}(\wedge^2 R)\cong\bigoplus_{\beta\in {Q_{-1}}(2(k-1))} S^\beta R$ and the middle map is the canonical one. 
\end{lemma}
\begin{proof}
Choose a basis $\{e_1,\ldots,e_n\}$ for $R$ and let $U$ be
  the unipotent subgroup of $\GL(R)$ given by upper triangular
  matrices with 1's on the diagonal, written in the basis
  $\{e_1,\ldots,e_n\}$. In other words $u\in U$ if and only if $u\cdot e_i=e_i+\sum_{j<i} \lambda_j e_j$
for $i=1,\ldots,r$.

The $U$-invariant vectors in $\wedge^k(\wedge^2 R)$ corresponding to
the decomposition  
\begin{equation}
\label{ref-A.6-57}
\wedge^k(\wedge^2 R)\cong\bigoplus_{\alpha\in {Q_{-1}}(2k)} S^\alpha R
\end{equation}
 were explicitly written down in
\cite[Prop. 2.3.9]{WeymanBook}. To explain this let $\alpha\in {Q_{-1}}(2k)$ and write it in
Frobenius coordinates as $(a_1,\ldots,a_u{;}a_1+1,\ldots,a_u+1)$. Then the highest weight
vector of the $S^\alpha R$-component in \eqref{ref-A.6-57} is given by
$u_\alpha:= \bigwedge_{i< j\le i+a_i} v_{ij}$ for $v_{ij}=e_i\wedge
e_j$ (we do not care about the sign of $u_\alpha$ so the ordering of the product is unimportant). If we represent $\alpha$
by a Young diagram then the index set of the exterior product
corresponds to the boxes strictly below the diagonal which makes it easy
to visualize why $u_\alpha$ is $U$-invariant and why it has weight $\alpha$
for the maximal torus corresponding of the diagonal matrices in $\GL(R)$.

We have $\phi(u_\alpha)=\sum_{ij}\pm v_{ij}\otimes \hat{u}_{\alpha,ij}$
where $\hat{u}_{\alpha,ij}$ is obtained from $u_{\alpha}$ by removing
the factor $v_{ij}$. Thus $\phi_{\alpha,\beta}(u_\alpha)
=\sum_{ij}\pm v_{ij}\otimes \pr_\beta(\hat{u}_{\alpha,ij})$ where $\pr_\beta:\wedge^{k-1}(\wedge^2 Q)
\r S^\beta R$ is the projection. Since the $v_{ij}$ are linearly independent in $\wedge^2 R$ it follows that
$\phi_{\alpha,\beta}(u_\alpha)$ can only be zero if $\pr_\beta(\hat{u}_{\alpha,ij})$ is zero
for all $i,j$. Now if $\beta\subset_2\alpha$ then there exist $i,j$ such that
$\hat{u}_{\alpha,ij}=\pm u_{\beta}$. Since by definition $\pr_\beta(u_\beta)=u_\beta\neq 0$
we obtain $\pr_\beta(\hat{u}_{\alpha,ij})\neq 0$ and thus also $\phi_{\alpha,\beta}(u_\alpha)\neq 0$.
\end{proof}

\def\cprime{$'$} \def\cprime{$'$} \def\cprime{$'$}
\providecommand{\bysame}{\leavevmode\hbox to3em{\hrulefill}\thinspace}
\providecommand{\MR}{\relax\ifhmode\unskip\space\fi MR }
\providecommand{\MRhref}[2]{%
  \href{http://www.ams.org/mathscinet-getitem?mr=#1}{#2}
}
\providecommand{\href}[2]{#2}

\end{document}